\documentclass[11pt]{amsart}
\usepackage{amssymb,latexsym}
\input xypic
  \xyoption{all}

\usepackage{amssymb,amsmath, amsthm}
\usepackage{ latexsym, amsfonts}

\usepackage{pifont,pxfonts,txfonts}
\usepackage{yfonts}
\usepackage{bbm}
\usepackage{calrsfs}
\usepackage{empheq}
\usepackage{color}
\usepackage[colorlinks,linkcolor=red,anchorcolor=blue,citecolor=green]{hyperref}
\usepackage{amstext, amsbsy, amscd}
\usepackage[mathscr]{eucal}

\usepackage{times}

\usepackage{amsfonts,extarrows}

   \newcommand{\ba}{\begin{eqnarray}}
   \newcommand{\na}{\end{eqnarray}}
   \newcommand{\ban}{\begin{eqnarray*}}
   \newcommand{\nan}{\end{eqnarray*}}
   \newcommand{\baln}{\begin{align*}}
   \newcommand{\naln}{\end{align*}}
\newtheorem{thm}{Theorem}[section]
\newtheorem{prop}[thm]{Proposition}
\newtheorem{cor}[thm]{Corollary}
\newtheorem{defi}[thm]{Definition}
\newtheorem{lem}[thm]{Lemma}

\newcommand{\hl}{\mathcal{H}}
\newcommand{\sll}{\mathcal{S}}
\newcommand{\el}{\mathcal{E}}
\newcommand{\lal}{\mathcal{L}}
\newcommand{\ul}{\mathcal{U}}
\newcommand{\cl}{\mathcal{C}}

\newcommand{\bl}{\mathcal{B}}
\newcommand{\ml}{\mathcal{M}}
\newcommand{\fl}{\mathcal{F}}

\newcommand{\rb}{\mathbb{R}}

\newcommand{~}{\quad}
\newcommand{\cb}{\mathbb{C}}

\newcommand{\zb}{\mathbb{Z}}
\newcommand{\pb}{\mathbb{P}}

\begin{document}

\title[Weinstein conjecture in product of symplectic manifolds]{The Weinstein Conjecture in Product of Symplectic Manifolds}

\author{Yanqiao Ding \& Jianxun Hu$^*$  }
\address{Department of Mathematics\\ Sun Yat-Sen University\\
                       Guangzhou, 510275 \\ P. R. China}
\email{dingyq6@mail2.sysu.edu.cn}

\address{Department of Mathematics\\ Sun Yat-Sen University\\
                       Guangzhou, 510275 \\ P. R. China}
\email{stsjxhu@mail.sysu.edu.cn}
\thanks{${}^*$Partially supported by the NSFC Grant  11371381}

\begin{abstract}
 In this paper, using pseudo-holomorphic curve method, one proves the Weinstein conjecture in the product $P_1\times P_2$ of two strongly geometrically bounded symplectic manifolds under some conditions with $P_1$. In particular, if $N$ is a closed manifold or a noncompact manifold of finite topological type, our result implies that the Weinstein conjecture in $\mathbb{C}\mathbb{P}^2\times T^*N$ holds.\\
Key words: Weinstein conjecture, $J$-holomorphic sphere, geometrically bounded.\\
Subject Classification: 53D35, 58D10.
\end{abstract}

\date{\today}
\maketitle
\tableofcontents

\section{Introduction }

Let $M$ be a symplectic manifold with symplectic form $\omega$.
A hypersurface $S\subset M$ is said to be of contact type if there exists a vector field $X$ defined on some neighborhood $U$ of $S$ such that
$(i)$ $X$ is transversal to $S$ and $(ii)$ $L_X\omega=\omega$.

For any hypersurface $S$ in symplectic manifold $M$, there exists a $1$-dimensional characteristic line bundle $\lal_S\subset TS$ defined by:
$$
\lal_S=\{ (x,\xi)\in T_xS|~\omega_x(\xi,\eta)=0, ~\forall\eta\in T_xS\}.
$$
Let $\xi$ be a section of the characteristic line bundle. The Weinstein conjecture claims that if $S$ is a compact hypersurface of contact type,
then $S$ carries at least one closed orbit of $\xi$, see \cite{w1979}.

In 1987, C. Viterbo \cite{v1987} proved the Weinstein conjecture for $(\rb^{2n},\omega_0)$ with the standard symplectic form $\omega_0$.
Later H. Hofer and C. Viterbo \cite{hv1988} showed the Weinstein conjecture was true for $(T^*M,-d\lambda)$,
where $\lambda$ was the Liouville form on the cotangent bundle $T^*M$ of a compact manifold $M$.
A. Floer, H. Hofer and C. Viterbo \cite{fhv1990} proved the stabilized Weinstein conjecture for $(P\times \cb^l,\omega\oplus\omega_0)$
under the assumption $[\omega]=0$ on $\pi_2(P)$. In 1992, H. Hofer and C. Viterbo \cite{hv1992} introduced the pseudo-holomorphic curve method
into the study of the Weinstein conjecture for some cases where the holomorphic spheres appeared.
They proved the Weinstein conjecture in $\cb\pb^n$, $S^2\times P$, if $P$ was a compact symplectic manifold with some conditions.
Lu \cite{lu1998} extended the results of H. Hofer and C. Viterbo to the strongly geometrically bounded (SGB) symplectic manifolds.
He showed the Weinstein conjecture holds in $S^2\times T^*N$, if $N$ was a closed manifold or a noncompact manifold of finite topological type.
G. Liu and G. Tian completely proved the stabilized version Weinstein conjecture in \cite{lt2000}.

Since the product of regular almost complex structures is not regular in general (see \cite{lu1997}),
the method of \cite{hv1992} can not be applied directly to any product manifolds.
Making use of the regularity criterion in \cite{ms2004}, we proved that there exists a regular almost complex structure, which is the product of regular almost complex structures, on the product of
some 4-dimensional manifolds and symplectic manifolds.  So this makes it possible to use the method of \cite{hv1992} to study the Weinstein conjecture for the product manifolds.
In this paper, one proves the Weinstein conjecture in the product $P_1\times P_2$ of two SGB symplectic manifolds under some conditions with $P_1$. In particular, if $N$ is a compact manifold or a noncompact manifold of finite topological type, our result implies that the Weinstein conjecture in $\mathbb{C}\mathbb{P}^2\times T^*N$ holds.

Next, we will introduce some notations and our result. Let $(V,\omega)$ be a symplectic manifold.
Let $\fl(V,\omega)$ be the space of all smooth almost complex structures which are compatible with $\omega$ on $(V,\omega)$.
The subset of regular almost complex structures (see Definition $\ref{defreg}$) in $\fl(V,\omega)$ is denoted by $\fl_{reg}(V,\omega)$.
For $J\in\fl(V,\omega)$, define $m(V,\omega,J)$ in $(0,+\infty]$ by
$$
m=\inf\{\langle\omega,[u]\rangle| \text{u is a nonconstant $J$-holomorphic sphere}\},
$$
where $\langle\omega,[u]\rangle=\int_{S^2}u^*\omega$ which depends only on the free homotopy class $[u]$ of $u$.
Define $m(V,\omega)\in[0,+\infty]$ by
\[
m(V,\omega)=\inf\{\langle\omega,\alpha\rangle|~\alpha\in[S^2,V],\langle\omega,\alpha\rangle>0\},
\]
where $[S^2,V]$ stands for the free homotopy classes.
A homotopy class $\alpha$ is said to be $\omega$-minimal if $m(V,\omega)=\langle\omega,\alpha\rangle$ and $\langle\omega,\alpha\rangle>0$.
Let $\alpha$ be an $\omega$-minimal homotopy class such that there exists a $J\in\fl(V,\omega)$ satisfies $m(V,\omega,J)=\langle\omega,\alpha\rangle$.
Define $\hl(\alpha,J,\Sigma_0,\Sigma_{\infty})$ to be the set of all $u\in C^\infty(S^2,V)$ such that
\begin{equation}\label{eqh}
[u]=\alpha,~u(*)\in\Sigma_*,~*\in\{0,\infty\},~\int_{|z|\leqslant1}u^*\omega=\frac{1}{2}\langle\omega,\alpha\rangle,~\bar\partial_Ju=0
\end{equation}
where $\Sigma_0,\Sigma_{\infty}$ are two disjoint smooth submanifolds of $V$ and closed as subsets.
We also assume that one of $\Sigma_0$ and $\Sigma_\infty$ is compact.

Under certain conditions, there are almost complex structures $\{\tilde J\}$,
which are as close as we want to $J$ with respect to $C^1$-topology, such that $\hl(\alpha,\tilde J,\Sigma_0,\Sigma_{\infty})$
is a smooth compact free $S^1$-manifold. Such a $\tilde J$ is called a regular almost complex structure at the situation
$(\alpha,\Sigma_0,\Sigma_{\infty})$.
Moreover, for any given regular $\tilde J_1$ and $\tilde J_2$ which are close to $J$ the compact smooth $S^1$-manifolds $\hl_1$ and $\hl_2$ belong
to the same free $S^1$-cobordism class. $d(\alpha, J,\Sigma_0,\Sigma_{\infty}):=[\hl(\alpha,\tilde J,\Sigma_0,\Sigma_{\infty})]$,
where $\tilde J$ is in a sufficiently small neighborhood of $J$.
The definition of $d(\alpha, J,\Sigma_0,\Sigma_{\infty})$ is not depend on the choice of $\tilde J$.

The following is our main result of this paper.
\begin{thm}\label{p4dsgb}
Let $(P_1,\omega_1)$, $(P_2,\omega_2)$ be two SGB symplectic manifolds with $dim P_1=4$.
$\alpha_1\in[S^2,P_1]$ is an $\omega_1$-minimal free homotopy class
which can be represented by an embedded $J_1$-holomorphic sphere such that
$$
0<\langle\omega_1,\alpha_1\rangle\leqslant m(P_2,\omega_2),
$$
where $J_1\in\fl_{reg}(P_1,\omega_1)$. $\Sigma^1_0,\Sigma^1_{\infty}$ are two disjoint nonempty compact submanifolds of $P_1$.
$\Sigma^2_0$ is a nonempty compact submanifold of $P_2$. Let $\Sigma_0=\Sigma^1_0\times\Sigma^2_0$, $\Sigma_{\infty}=\Sigma^1_{\infty}\times P_2$.
Suppose that there is a smooth Hamiltonian $H:P_1\times P_2\to\rb$ such that
\[
H|_{\ul(\Sigma_0)}\equiv h_0, H|_{\ul(\Sigma_{\infty})}\equiv h_{\infty},h_0<h_{\infty} \text{ and } h_0\leqslant H\leqslant h_{\infty}.
\]
Where the open neighborhoods $\ul(\Sigma_0)$ and $\ul(\Sigma_{\infty})$ are disjoint and such that
\begin{equation}\label{eqkcpt}
K:=(P_1\times P_2)\setminus(\ul(\Sigma_0)\cup\ul(\Sigma_{\infty}))\text{ is compact}.
\end{equation}
Then if $d(\alpha_1,J_1,\Sigma^1_0,\Sigma^1_{\infty})\neq[\varnothing]$,
the Hamiltonian system $\dot{x}=X_H(x)$ possesses a nonconstant $T$-periodic solution $x=x(t)$ with
\[
0<T(h_{\infty}-h_0)<\langle\omega_1,\alpha_1\rangle, h_0\leqslant H(x(t))\leqslant h_{\infty}.
\]
\end{thm}

As in \cite{hv1992}, it is easy to prove the Weinstein conjecture in $P_1\times P_2$ from Theorem $\ref{p4dsgb}$.
\begin{cor}
Let $(P_1\times P_2,\omega_1\oplus\omega_2)$, $\alpha_1\in[S^2,P_1]$, $J_1\in\fl_{reg}(P_1,\omega_1)$,
$\Sigma^1_0$, $\Sigma^1_{\infty}$, $\Sigma_0$, $\Sigma_\infty$
satisfy the hypothesis of Theorem $\ref{p4dsgb}$.
Then any stable compact smooth hypersurface $\sll$ in $(P_1\times P_2,\omega_1\oplus\omega_2)$ separating $\Sigma_0$
from $\Sigma_\infty$ possesses at least one periodic Hamiltonian trajectory.
\end{cor}

\begin{cor}\label{cormain}
Let $u:S^2=\cb\cup\{\infty\}\hookrightarrow\cb\pb^2$ be a holomorphic embedding and $\{x\}$,$\{y\}$ two different points in the image of $u$.
Suppose $N$ is a closed manifold or a manifold of finite topological type.
Define $\Sigma_0=\{p_0\}$, $p_0\in\{x\}\times T^*N$, $\Sigma_\infty=\{y\}\times T^*N$.
Then any stable compact smooth hypersurface $\sll$ in $\cb\pb^2\times T^*N$ separating $\Sigma_0$
from $\Sigma_\infty$ possesses at least one periodic Hamiltonian trajectory.
\end{cor}

\section{Product of Regular Almost Complex Structures}\label{secj}

Let $(V,\omega)$ be a symplectic manifold and $J\in\fl(V,\omega)$. For a smooth map $u:S^2\to V$,
the space of smooth vector fields $\xi(z)\in T_{u(z)}V$ along $u$ will be denoted by $\Omega^0(S^2,u^*TV)$
and the space of smooth $J$-antilinear $1$-forms on $S^2$ with values in $u^*TV$ by $\Omega^{0,1}(S^2,u^*TV)$.
Then the vertical differential of $\bar\partial_J(u)$, $D_u:\Omega^0(S^2,u^*TV)\to\Omega^{0,1}(S^2,u^*TV)$,
have the following expression:
\begin{equation}\label{edu}
D_u\xi=\frac{1}{2}(\nabla\xi+J(u)\nabla\xi\circ i)-\frac{1}{2}J(u)(\nabla_\xi J)(u)\partial_J(u),~\forall\xi\in\Omega^0(S^2,u^*TV)
\end{equation}
where $\partial_J(u):=\frac{1}{2}(du-J\circ du\circ i)$ and $\nabla$ denotes the Levi-Civita connection of the metric $\omega(\cdot,J\cdot)$.

A $J$-homomorphic sphere $u:S^2\to V$ is said to be {\bf multiply covered} if there exists a $J$-holomorphic sphere $u':S^2\to V$,
and a holomorphic branched covering $\phi:S^2\to S^2$ such that
\[
u=u'\circ\phi,~deg(\phi)>1.
\]
The curve $u$ is called {\bf simple} if it is not multiply covered.

\begin{defi}\label{defreg}
An almost complex structure $J$ on $V$ is called regular at the situation $(\alpha,\Sigma_0,\Sigma_{\infty})$,
if for every $u\in H^{2,2}(S^2,V)$ which satisfies condition $(\ref{eqh})$ $D_u$ is onto.
In particular for a regular $J$ the set $\hl(\alpha,J,\Sigma_0,\Sigma_{\infty})$ is a smooth $S^1$-manifold.
\end{defi}

{\bf Remark}: By elliptic regularity theory every $u\in H^{2,2}(S^2,V)$ which satisfies condition $(\ref{eqh})$ is smooth.

There is a regularity criterion in \cite{ms2004} which is very important for us.
\begin{lem}[Lemma $3.3.2$ in \cite{ms2004}]\label{lem322}
Let $E\to S^2$ be a complex vector bundle of rank $n$ and
\[
D:\Omega^0(S^2,E)\to\Omega^{0,1}(S^2,E)
\]
be a real linear Cauchy-Riemann operator. Suppose that there exists a splitting $E=L_1\oplus\cdot\cdot\cdot\oplus L_n$
into complex line bundles such that each subbundle $L_1\oplus\cdot\cdot\cdot\oplus L_k$, $k=1,...,n$, is invariant under $D$.
Then $D$ is surjective if and only if $c_1(L_k)\geqslant-1$ for every $k$.
\end{lem}

{\bf Remark}: $\Omega^0(S^2,E)$ denotes the space of all smooth vector fields $\xi(z)\in E_z$.
$\Omega^{0,1}(S^2,E)$ denotes the space of smooth $J$-antilinear $1$-forms on $S^2$ with values in $E$.
Let $\pi_k:E\to L_k$ denote the projection onto the $k$th summand. Then the subbundle $L_1\oplus\cdot\cdot\cdot\oplus L_k$
is invariant under $D$ means that if $i>k$, $\pi_i(D\xi_j)=0$,$\forall\xi_j\in\Omega^0(S^2,L_j),j=1,...,k$.
Here and throughout this section we identify the first Chern class $c_1(L)$ of $L$ with the corresponding Chern number
$\langle c_1(L),[S^2]\rangle$.

{\bf Remark}: The operator $D_u$ is obviously a real linear Cauchy-Riemann operator.

Using Lemma $\ref{lem322}$, we can give a sufficient condition which guarantees a product regular almost complex structure is still regular.
First, we will introduce some notations.
The number of all self-intersections of a curve $u$ will be denoted by
\[
\delta(u):=\frac{1}{2}\#\{(z_0,z_1)\in\Sigma\times\Sigma|z_0\neq z_1,u(z_0)=u(z_1)\}.
\]
We denote by $c_1(A)=\langle c_1(TM),A\rangle$ for $A\in H_2(M;\zb)$, where $c_1(TM)$ is the first Chern class of $TM$,
by $A_0\cdot A_1$ the intersection number of two classes $A_0$ and $A_1$,
and by $\chi(\Sigma)$ the Euler characteristic of a closed Riemann surface $\Sigma$.

\begin{lem}[adjunction inequality in \cite{ms2004}]
Let $(M,J)$ be an almost complex $4$-manifold and $A\in H_2(M;\zb)$ be a homology class that is represented
by a simple $J$-holomorphic curve $u:\Sigma\to M$.
Then
\[
2\delta(u)-\chi(\Sigma)\leqslant A\cdot A-c_1(A).
\]
with equality if and only if $u$ is an immersion with only transverse self-intersections
(i.e. if $z_0\neq z_1$ and $u(z_0)=u(z_1)=:x$, then $T_xM=im du(z_0)\oplus im du(z_1)$).
\end{lem}

For the $4$-manifolds, we have the following Proposition.
\begin{prop}\label{p4dj}
Let $(P_1,\omega_1)$ be a symplectic $4$-manifold and $(P_2,\omega_2)$ a symplectic manifold.
Assume $\alpha_1\in[S^2,P_1]$ is an $\omega_1$-minimal free homotopy class which
can be represented by an embedded $J_1$-holomorphic sphere $u$ such that
$$
0<\langle\omega_1,\alpha_1\rangle\leqslant m(P_2,\omega_2).
$$
$\Sigma^1_0,\Sigma^1_{\infty}$ are two disjoint nonempty compact submanifolds of $P_1$.
$\Sigma^2_0$ is a nonempty compact submanifold of $P_2$.
Let $\Sigma_0=\Sigma^1_0\times\Sigma^2_0$, $\Sigma_{\infty}=\Sigma^1_{\infty}\times P_2$ and $\alpha\in[S^2,V:z\mapsto(\alpha_1(z),p_0)]$,
where $p_0\in\Sigma^2_0$.
If $J_1$ is regular at the situation $(\alpha_1,\Sigma_0^1,\Sigma_{\infty}^1)$ in $P_1$ and $J_2\in\fl(P_2,\omega_2)$,
then the product almost complex structure $J=J_1\times J_2$ is regular at the situation $(\alpha,\Sigma_0,\Sigma_{\infty})$ in $P_1\times P_2$.
\end{prop}

\begin{proof}
First it is easy to see every $J_1$-holomorphic sphere $u$ which represents $\alpha_1$ is simple.
In fact, if $u$ is multiply covered there exists a $J_1$-holomorphic sphere $u':S^2\to P_1$,
and a holomorphic branched covering $\phi:S^2\to S^2$ such that
\[
u=u'\circ\phi,~deg(\phi)=k>1.
\]
Evidently $\langle\omega_1,[u']\rangle=\frac{1}{k}\langle\omega_1,\alpha_1\rangle$ since $\alpha_1=[u]$. Hence
$$
\langle\omega_1,\alpha_1\rangle\leqslant m(P_1,\omega_1,J_1)\leqslant\frac{1}{k}\langle\omega_1,\alpha_1\rangle,
$$
giving a contradiction to our assumption that $\alpha_1$ is $\omega_1$-minimal.
Assume $u$ represents the homology class $A\in H^2(P_1;\zb)$, i.e. $u_*([S^2])=A$.
Then all the $J_1$-holomorphic spheres represent $\alpha_1$ will represent $A$.

Since $A\in H^2(P_1;\zb)$ is represented by an embedded $J_1$-holomorphic sphere $u$ which is also simple, by the adjunction inequality we can get
\[
-2=A\cdot A-c_1(A).
\]
For every simple $J_1$-holomorphic sphere $v:S^2\to P_1$ which represents $A$, we have
\begin{align*}
2\delta(v)-2&\leqslant A\cdot A-c_1(A),\\
2\delta(v)&\leqslant0,\\
\delta(v)&=0.
\end{align*}
The equality of the adjunction inequality holds for $v$.
Thus every simple $J_1$-holomorphic sphere $v$ which represents $A$ is an embedded curve.
We can get every $J_1$-holomorphic sphere represents $\alpha_1$ is an embedded curve.

Assume $\tilde u\in H^{2,2}(S^2,P_1\times P_2)$ and $\tilde u$ satisfies
$$
[\tilde u]=\alpha,~\tilde u(*)\in\Sigma_*,~*\in\{0,\infty\},~
\int_{|z|\leqslant1}\tilde u^*\omega=\frac{1}{2}\langle\omega,\alpha\rangle,~\bar\partial_J\tilde u=0,
$$
where $\omega=\omega_1\oplus\omega_2$.
The $J$-holomorphic $\alpha$ sphere has the form $\tilde u(z)=(u(z),p_0)$, where $u\in H^{2,2}(S^2,P_1)$ and satisfies
$$
[u]=\alpha_1,~u(*)\in\Sigma_*^1,~*\in\{0,\infty\},~\int_{|z|\leqslant1}u^*\omega_1=\frac{1}{2}\langle\omega_1,\alpha_1\rangle,~\bar\partial_{J_1}u=0.
$$
We have the splitting
\baln
\tilde u^*T(P_1\times P_2)&=u^*TP_1\oplus(S^2\times T_{p_0}P_2)\\
&=u^*TP_1\oplus L_2\oplus...\oplus L_{n+1}.
\end{align*}
It follows from the definition of $D_u$ $(\ref{edu})$ that
\[
D_u(du\circ\zeta)=du\circ\bar\partial_j\zeta
\]
for every vector field $\zeta\in \text{Vect}(S^2)$. For the embedded curve $u$, the complex subbundle
\[
L_0:=im(du)\subset u^*TP_1
\]
is invariant under $D_u$.
Now let $L_1\subset u^*TP_1$ be the orthogonal complement of $L_0$ with respect to any Hermitian inner product of $u^*TP_1$.
Then by Lemma $\ref{lem322}$
\[
u^*TP_1=L_0\oplus L_1,~ c_1(L_0)\geqslant-1,~ c_1(L_1)\geqslant-1,
\]
because $J_1$ is regular at the situation $(\alpha_1,\Sigma_0^1,\Sigma_{\infty}^1)$ in $P_1$.
In the product manifold $(P_1\times P_2,\omega_1\oplus\omega_2)$, $\omega_1(\cdot,J_1\cdot)+\omega_2(\cdot,J_2\cdot)$
defines a product metric on $P_1\times P_2$. Let $\nabla$ be the Levi-Civita connection on $P_1\times P_2$ and
$\nabla^i$ the Levi-Civita connection on $P_i$, $i=1,2$, respectively.
By the relation between $\nabla$ and $\nabla^i$, $i=1,2$,
we know in the product manifold $P_1\times P_2$
\[
D_{\tilde u}\xi_j=D_u\xi_j,~\forall\xi_j\in\Omega^0(S^2,L_j), j=0,1.
\]
Thus the subbundles $L_0$, $L_0\oplus L_1$, are invariant under $D_{\tilde u}$ too.
In the trivial bundle $S^2\times T_{p_0}P_2$, each subbundle $L_{2}\oplus...\oplus L_{1+j}$, $j=1,...,n$, is obviously invariant under $D_{\tilde u}$.
$c_1(L_j)\geqslant-1$, $j=2,...,n+1$. By Lemma $\ref{lem322}$ again, we know $D_{\tilde u}$ is surjective.
\end{proof}

{\bf Remark}: From the arguments of Lemma $3.3.3$, Corollary $3.3.4$ and Corollary $3.3.5$ in \cite{ms2004}, we can get the above Proposition easily.

\section{Holomorphic Spheres}\label{defgb}

Let us recall the definition of geometrically bounded manifold (cf.\cite{alp1994}, \cite{g1985}, \cite{lu1998}).

\begin{defi}\label{dgb}
Let $(M,\omega)$ be a symplectic manifold without boundary.
we will call it geometrically bounded if there exists an almost complex structure $J$
and a complete Riemannian metric $g$ on $M$ such that the following properties are satisfied:\\
$1.$ $J$ is uniformly tamed by $\omega$; that is, there exist strictly positive constants $\alpha$ and $\beta$ such that
\[
\omega(X,JX)\geqslant\alpha\parallel X\parallel_g^2~\text{and}~|\omega(X,Y)|\leqslant\beta\parallel X\parallel_g\parallel Y\parallel_g
\]
for all $X,Y\in TM$;\\
$2.$ the sectional curvature $K_g\leqslant C$(a positive constant) and the injectivity radius $i(M,g)>0$.
\end{defi}

\begin{defi}[Definition $2.4$ in \cite{lu1998}]
In Definition $\ref{dgb}$ if we require $J\in\fl(M,\omega)$,
then the symplectic manifold $(M,\omega)$ is called strongly geometrically bounded (SGB).
\end{defi}

It is well known that the closed symplectic manifolds are SGB and a product of two SGB symplectic manifolds is SGB.
It is easy to prove the symplectic manifolds which at infinity are
isomorphic to the symplectization of a closed contact manifold are SGB (cf.\cite{cgk2004}).
The standard cotangent bundles as well as the twisted cotangent bundles over closed manifolds are SGB
(\text{cf.} \cite{cgk2004}, \cite{lu1998}).

Let $(P_1,\omega_1,J_1,g_1)$, $(P_2,\omega_2,J_2,g_2)$ be two SGB symplectic manifolds such that dim$P_1=4$.
$V=P_1\times P_2$, $\omega=\omega_1\oplus\omega_2$, $J=J_1\times J_2$, $g=g_1\oplus g_2$.
Then $(V,\omega,J,g)$ is a SGB symplectic manifold.
Assume $m(V,\omega,J)<\infty$, and let $\alpha\in[S^2,V:z\mapsto(\alpha_1(z),p_0)]$, $p_0\in\Sigma_0^2$, be a free homotopy class
which is defined in Proposition $\ref{p4dj}$ such that
\begin{equation}\label{eqhomology}
\langle\omega,\alpha\rangle=m(V,\omega,J).
\end{equation}
From the definition of $m(V,\omega,J)$, we can get that a $J$-holomorphic sphere which represents $\alpha$ is simple.

Consider the Banach manifold $\bl$ consisting of all maps $u\in H^{2,2}(S^2,V)$ such that with $D=\{z||z|\leqslant1\}$
\[
[u]=\alpha,~u(*)\in\Sigma_*,~*\in\{0,\infty\},~\int_Du^*\omega=\frac{1}{2}\langle\omega,\alpha\rangle,
\]
where $\Sigma_0,\Sigma_{\infty}$ are two disjoint smooth submanifolds without boundary of $V$ and closed as subsets in $V$.
We also assume that one of $\Sigma_0$ and $\Sigma_\infty$ is compact.
Denote by $\bar{X}_J\to S^2\times V$ the vector bundle whose fiber over $(z,v)\in S^2\times V$
consists of all linear maps $\phi:T_zS^2\to T_vV$ such that $J(v)\phi=-\phi\circ i$.
Given $u:S^2\to V$ we denote by $\bar u:S^2\to S^2\times V$ the "graph map" $\bar u(z)=(z,u(z))$ and
write $\bar u^*\bar{X}_J\to S^2$ for the pull back bundle. Let $\el$ be the Banach bundle $\el\to\bl$
whose fiber $\el_u=H^{1,2}(\bar u^*\bar X_J)$ at $u\in H^{2,2}(S^2,V)$ consists of all $H^{1,2}$ sections of $\bar u^*\bar X_J\to S^2$.
The nonlinear Cauchy Riemann operator $\bar{\partial}_J$, $\bar{\partial}_Ju=du+J\circ du\circ i$,
can be considered as a smooth section of $\el\to\bl$, and its zero set is $\hl(\alpha,J,\Sigma_0,\Sigma_{\infty})$.
By elliptic regularity theory every $u\in\bl$ with $\bar\partial_Ju=0$ is smooth.
H.Hofer and C.Viterbo proved some propositions-Propositions $2.3$, $2.4$ and $2.7$ in \cite{hv1992}-
for the compact manifold $V$ which guaranteed the d-index was well defined and made the existence of closed orbit possible.
Lu proved a prior compactness property (Proposition $2.5$ in \cite{lu1998})
for the SGB symplectic manifold. Utilizing the prior compactness and the assumption $(\ref{eqkcpt})$, Lu \cite{lu1998} showed
the Propositions $2.3$ and $2.4$ in \cite{hv1992} also held true for the case of SGB symplectic manifold
if the neighborhood $U_J$ and $\fl_{reg}(V,\omega)\cap U_J$ of $J$ in these Propositions were replaced by
$\ul(J,\delta,f_{r_0})$ and $\fl_{reg}(V,\omega)\cap\ul(J,\delta,f_{r_0})$.
The definition of $\ul(J,\delta,f_{r_0})$ is given in \cite{lu1998}. In the following, $\ul(J,\delta,f_{r_0})$ is abbreviated to $\ul$.
So the d-index $d(\alpha, J,\Sigma_0,\Sigma_{\infty}):=[\hl(\alpha,\tilde J,\Sigma_0,\Sigma_{\infty})]$
is well defined in the SGB symplectic manifold.

\begin{prop}
Let $(V,\omega)$ be a SGB symplectic manifold, $J\in\fl(V,\omega)$, $m(V,\omega,J)=\langle\omega,\alpha\rangle$,
Let $\Sigma_0,\Sigma_\infty$ be described above, then there exists an open neighborhood $\ul$ of $J$ such that \\
$(1)$ For all $\tilde J\in\fl_{reg}(V,\omega)\cap\ul$, the set $\hl(\alpha,\tilde J,\Sigma_0,\Sigma_{\infty})$ is a compact smooth $S^1$-manifold.\\
$(2)$ $\fl_{reg}(V,\omega)\cap\ul$ is dense in $\ul$.\\
$(3)$ Let $J_0'$, and $J_1'$ be close to $J$ in $\ul$, and $J_0', J_1'\in\fl_{reg}(V,\omega)$.
Suppose $\lambda\to J_\lambda'$ is a smooth homotopy with $\lambda\in[0,1]$ and $J_\lambda'\in\ul$.
Then there exists a smooth arbitrarily small perturbation of $[\lambda\to J_\lambda']$
with the end points fixed, say $[\lambda\to\tilde J_\lambda]$, such that
\[
\ml:=\{(\lambda,u)\in[0,1]\times\bl|\bar\partial_{\tilde J_\lambda}u=0\}
\]
is a compact $S^1$-manifold with boundary
\[
\partial\ml=\hl(\alpha, J_0',\Sigma_0,\Sigma_{\infty})\amalg\hl(\alpha, J_1',\Sigma_0,\Sigma_{\infty}).
\]
\end{prop}

Let $H:V\to\rb$ be a smooth map and $g_J(\cdot,\cdot)=\omega(\cdot,J\cdot)$ the Riemannian metric.
We denote by $\nabla H$ the gradient of $H$ with respect to the metric $g_J$.
For suitable neighborhoods $\ul(\Sigma_0),\ul(\Sigma_{\infty})$ of $\Sigma_0,\Sigma_{\infty}$ respectively,
suppose $H|_{\ul(\Sigma_0)}\equiv h_0$, $H|_{\ul(\Sigma_{\infty})}\equiv h_{\infty}$, $h_0<h_{\infty}$ and $h_0\leqslant H\leqslant h_{\infty}$.
Consider the subset
\[
W=[(S^2\backslash\{0,\infty\})\times V]\cup[\{0\}\times\ul(\Sigma_0)]\cup[\{\infty\}\times\ul(\Sigma_{\infty})].
\]
We define a section $\hat{h}$ of $\bar X_J|_W$ associated to $H$ by
\[
\hat h:W\to\bar X_J,~\hat h(z,v)=:\phi.
\]
Where $\phi$ is the unique complex antilinear map $T_zS^2\to T_vV$ satisfying the following:\\
$1.$ If $z=0$ or $\infty$, $\phi$ is the zero map,\\
$2.$ If $z\neq0$ and $\neq\infty$, $\phi$ maps the tangent vector $z\in T_zS^2=\cb$ to $\frac{1}{2\pi}\nabla H(v)$.
Here we took the identity chart $S^2\supset\cb\simeq\cb$ to distinguish in $T_zS^2$ for $z\in\cb$ the tangent vector $z$.

If $u\in\bl$ then the associated graph map $\bar u$, $\bar u(z)=(z,u(z))$ maps $z\in S^2$ into $W\subset S^2\times V$.
Consequently we can define $h(u)\in\el$ by
\[
h(u)(z)=\hat h(z,u(z)).
\]
Now we define a parameter depending family of smooth section of $\el\to\bl$ by
\[
f_{\lambda}(u)=\bar{\partial}_{J}u+\lambda h(u).
\]
Clearly, $f_\lambda$ is $S^1-$equivalent for every $\lambda$ and $f_{\lambda}$ is a Fredholm section in the sense
that at every zero $u$ of $f_\lambda$ the linearisation
$Df_\lambda:T_u\bl\to\el_u$ is Fredholm. Consider the set
\[
\cl=\{(\lambda,u)\in[0,+\infty)\times\bl|f_\lambda(u)=0\}.
\]
By elliptic regularity theory, $\cl\subset[0,+\infty)\times C^\infty(S^2,V).$ let $\cl_\lambda=\{u|(\lambda,u)\in\cl\}$.
Then $\cl_0$ is a compact smooth manifold with a free smooth $S^1$-action, and $\cl_0=\hl(\alpha,J,\Sigma_0,\Sigma_{\infty})$.
Lu \cite{lu1998} showed that if the manifold $V$ is SGB,
the Proposition $2.7$ in \cite{hv1992} was also true.

\begin{prop}[PROPOSITION $3.1$ in \cite{lu1998}]\label{plu31}
Let $\alpha\in[S^2,V]$, $\Sigma_0,\Sigma_\infty,J$ and $H$ be as above, and let $\cl$ be compact. Then
\[
d(\alpha,J,\Sigma_0,\Sigma_\infty)=[\varnothing],
\]
i.e. $\hl(\alpha,J,\Sigma_0,\Sigma_\infty)$ is the boundary of a smooth compact manifold $\ml$ equipped with a free $S^1$-action,
so that the action on $\partial\ml$ coincides with the action on $\hl$.
\end{prop}

As in \cite{hv1992} and \cite{lu1998}, we have the following Proposition:
\begin{prop}\label{pmain}
Let $(V,\omega)$ be a SGB symplectic manifold. $\Sigma_0$,$\Sigma_\infty$ are described above.
$J\in\fl(V,\omega)$ such that $m(V,\omega,J)\geqslant\langle\omega,\alpha\rangle$, where $\alpha\in[S^2,V]$.
Let $\el\to\bl$ be the Hilbert space bundle defined above. Let $H:V\to\rb$ be a smooth map such that
$$
H|_{\ul(\Sigma_0)}\equiv h_0, H|_{\ul(\Sigma_{\infty})}\equiv h_{\infty},h_0<h_{\infty} \text{ and } h_0\leqslant H\leqslant h_{\infty}.
$$
Let $\cl$ be defined above. Then \\
$(1)$ If $(\lambda,u)\in\cl$, then $\lambda\in[0,\lambda_\infty]$,$\lambda_\infty=(h_\infty-h_0)^{-1}\langle\omega,\alpha\rangle$;\\
$(2)$ For every multi index $\beta$ there is a constant $C_\beta>0$ such that for every $(\lambda,u)\in\cl$,
$v=u\circ\phi$, here $\phi:S^1\times\rb\to\cb,\phi(t,s)=e^{2\pi(s+it)}$.
\[
|(D^\beta v)(x)|\leqslant C_\beta, ~\forall x\in S^1\times\rb
\]
$(3)$ There exists $\varepsilon>0$ such that for every $(\lambda,u)\in\cl$ we have:
if $v(s)(S^1)\not\subset\ul(\Sigma_0)$ then
\[
\int^s_{-\infty}\int^1_0v^*\omega-\lambda\int_0^1H(v(s)(t))dt\geqslant\varepsilon-\lambda h_0.
\]
If $v(s)(S^1)\not\subset\ul(\Sigma_\infty)$ then
\[
\int^s_{-\infty}\int^1_0v^*\omega-\lambda\int_0^1H(v(s)(t))dt\leqslant\langle\omega,\alpha\rangle-\varepsilon-\lambda h_\infty.
\]
\end{prop}

\begin{proof}[{Sketch of the proof}]
From Theorem $2.9$ in \cite{lu1998}, we obtain that $\cup_{(\lambda,u)\in\cl}u(S^2)$ is contained in a compact subset of $V$.
Following almost the same arguments of Theorem $3.4$ in \cite{hv1992}, we can see that the proposition is also true.
\end{proof}

\section{Proof of Main Theorem}

The $J$-holomorphic sphere method always requires the regular almost complex structure.
In order to get the relation between the $d$-index of $P_1\times P_2$ with the $d$-index of $P_i$, $i\in\{1,2\}$,
we need a regular almost complex structure $J=J_1\times J_2$, where $J_i\in\fl_{reg}(P_i,\omega_i)$, $i\in\{1,2\}$.
However, the product of regular almost complex structures is not regular in general.
Thus Proposition $\ref{p4dj}$ is necessary for our case. Now we can prove Theorem $\ref{p4dsgb}$.
\begin{proof}[{\bf Proof of Theorem $\ref{p4dsgb}$}]
Let $\alpha\in[S^2,P_1\times P_2]$ be of the form $[S^2\to P_1\times P_2: z\mapsto(\alpha_1(z),p_0)]$,
where $p_0\in \Sigma^2_0$ is a fixed point and $\alpha_1\in[S^2,P_1]$ is defined in the hypothesis of Theorem $\ref{p4dsgb}$.
On $P_1\times P_2$ we take the product almost
complex structure $J=J_1\times J_2$, where $J_1\in\fl_{reg}(P_1,\omega_1),J_2\in\fl_{reg}(P_2,\omega_2)$. Then
\begin{align*}
\langle\omega_1\oplus\omega_2,\alpha\rangle&=m(P_1\times P_2,\omega_1\oplus\omega_2,J)\\
                          &=m(P_1,\omega_1,J_1)=\langle\omega_1,\alpha_1\rangle\leqslant m(P_2,\omega_2).
\end{align*}
By Proposition $\ref{p4dj}$, $J$ is regular at the situation $(\alpha,\Sigma_0,\Sigma_{\infty})$.
From $d(\alpha_1,J_1,\Sigma^1_0,\Sigma^1_{\infty})\neq[\varnothing]$ and $m(P_1\times P_2,\omega_1\oplus\omega_2,J)\leqslant m(P_2,\omega_2)$,
we have $d(\alpha,J,\Sigma_0,\Sigma_{\infty})\neq[\varnothing]$.

In the following, we use the idea of \cite{hv1992} to prove Theorem $\ref{p4dsgb}$.
From Proposition $\ref{plu31}$, we can get $\cl$ is noncompact. We can assume $\{(\lambda_k,u_k)\}\subset\cl$ such that
\[
\lambda_k\to\lambda, \{(\lambda_k,u_k)\} \text{ has no convergent subsequence}.
\]
For every $(\lambda, u)\in\cl$, We define $v=u\circ\phi$, where $\phi:S^1\times\rb\to\cb,\phi(t,s)=e^{2\pi(s+it)}$.
Define $a:\rb\to\rb$ as follows, $a(s):=\int_{(-\infty,s]\times S^1}v^*\omega-\int_0^1\lambda H(v(s,t))dt$,
where $\omega=\omega_1\oplus\omega_2$.
From Proposition $\ref{pmain}$, we have
\begin{align*}
\langle\omega_1,\alpha_1\rangle-\lambda(h_\infty-h_0)&=\int_{-\infty}^\infty a'(s)ds\\
&=\int_{-\infty}^{+\infty}\int_0^1|v_s|^2dsdt\\
&\geqslant\int_{-\infty}^{s_0(v)}\int_0^1v^*\omega\geqslant\varepsilon,
\end{align*}
where $s_0(v)=\sup\{s|v((-\infty,s]\times S^1)\subset \ul(\Sigma)\}$. The last inequality is proved by the Lemma $3.1$ in \cite{hv1992},
which is also true here.

If $(\lambda,u)\in\cl$, then
\begin{align*}
0\leqslant\lambda&\leqslant(h_\infty-h_0)^{-1}(\langle\omega_1,\alpha_1\rangle-\varepsilon)\\
&=:(h_\infty-h_0)^{-1}\langle\omega_1,\alpha_1\rangle-\varepsilon'\\
&=\lambda_\infty-\varepsilon'.
\end{align*}
We define two sequences of numbers by
\begin{align*}
&s_k^0=\sup\{s|v_k((-\infty,s]\times S^1)\subset\ul(\Sigma_0)\},\\
&s_k^\infty=\inf\{s|v_k([s,+\infty)\times S^1)\subset\ul(\Sigma_\infty)\}.
\end{align*}
Note that $v_k$ denotes the map induced by $u_k$ on the cylinder. Clearly $s_k^0\leqslant s_k^\infty$.

Now we will show $s_k^\infty-s_k^0\to\infty.$ Arguing indirectly we may assume after taking a subsequence that for some constant $b>0$
\[
|s_k^\infty-s_k^0|\leqslant b~~~\text{for all}~k.
\]
Let $\hat u_k(z)=u_k(s_kz),s_k>0.$ Replacing $u_k$ by $\hat u_k,$ we may assume that for some positive constant $c>0$,
\begin{equation}\label{e1}
-c\leqslant\hat s_k^0\leqslant\hat s_k^\infty\leqslant c,
\end{equation}
where $\hat s_k^0$, $\hat s_k^\infty$ are the sequences associated to $\hat u_k.$
From (\ref{e1}) and the previous discussion, it follows immediately that $\{\hat u_k\}$ has a convergent subsequence in $H^{2,2}(S^2,P_1\times P_2)$,
say $\hat u_k\to u$, where $u$ satisfies
\begin{equation}\label{f1}
\left\{ \begin{aligned}
         \bar\partial_Ju+&\lambda h(u)=0\\
[u]&=\alpha\\
u(0)&\in\Sigma_0\\
u(\infty)&\in\Sigma_\infty.
        \end{aligned} \right.
\end{equation}
In fact, since $(\ref{e1})$ holds, the nonlinearity $u\to h(u)$ is well behaved and one can use
\emph{Bubble off analysis} to obtain the solution $u$ of $(\ref{f1})$.
\[
\frac{1}{2}\langle\omega_1\oplus\omega_2,\alpha\rangle=\frac{1}{2}\langle\omega_1,\alpha_1\rangle=\int_Du_k^*\omega=\int_{s_k^{-1}D}\hat u_k^*\omega,
\]
where $\omega=\omega_1\oplus\omega_2$. If $s_k\to0$ or $+\infty$, since $\hat u_k^*\omega\to u^*\omega$ in $H^{1,1}(S^2,\rb)$,
we have
\[
\frac{1}{2}\langle\omega_1,\alpha_1\rangle=0,~\text{or}~\frac{1}{2}\langle\omega_1,\alpha_1\rangle=\langle\omega_1,\alpha_1\rangle.
\]
This contradiction shows that $s_k\in(a,\frac{1}{a})$ for all $k$ for some suitable $a>0$ independent of $k$.
Hence, from the definition of $\hat u_k$ and the fact that $\hat u_k\to u$ it follows that $\{u_k\}$ is convergent itself.
However, this contradicts our assumption on $\{(\lambda_k,u_k)\}$. Therefore we know that
\[
s_k^\infty-s_k^0\to+\infty.
\]
We have
\[
\int_{s_k^0}^{s_k^\infty}\int_0^1|-J(v_k)\frac{\partial v_k}{\partial t}-\lambda_k\nabla H(v_k)|^2dtds
\leqslant\langle\omega_1,\alpha_1\rangle-2\varepsilon-\lambda_k(h_\infty-h_0).
\]
Hence, we can find a sequence $\{s_k\}$,
\[
s_k\in[s_k^0,s_k^\infty],
\]
such that with $x_k:=v_k(s_k,\cdot)$
\[
\parallel -J(x_k)\dot x_k-\lambda_k\nabla H(x_k)\parallel_{L^2(x_k^*T(P_1\times P_2))}\to0.
\]
Eventually taking a subsequence we may assume
\begin{equation}\label{f2}
\left\{ \begin{aligned}
         x_k&\to x ~~~~~\text{in}~H^1(S^1,P_1\times P_2)\\
-J(x)\dot x&-\lambda\nabla H(x)=0.
        \end{aligned} \right.
\end{equation}
It is obvious that $x\in C^\infty(S^1,P_1\times P_2)$. We first assume $\lambda=0.$ Let $\tilde u_k:S^2\to P_1$ be the map induced from $u_k:S^2\to P_1\times P_2$ by the projection onto the first factor. Then,
\[
\langle\omega_1,\alpha_1\rangle=\int_{S^2}u_k^*\omega_1.
\]
Now let $\tilde v_k:Z\to P_1$ be the map induced from $\tilde u_k$ in the cylinder.
Since $\nabla H$ vanishes on $\Sigma_0$ and $\Sigma_\infty$,
$\tilde u_k$ is holomorphic in the neighbourhood of all $z$ such that $\tilde u_k(z)$ is close to $\Sigma_0^1$ or $\Sigma_\infty^1$.

If $(\ref{f2})$ holds, we can use $\tilde v_k:(-\infty,s_k]\times S^1\to P_1$ and $\tilde v_k:[s_k,+\infty)\times S^1\to P_1$ to construct maps
\[
g_{\pm\infty}^k:S^2\to P_1
\]
such that
\begin{align*}
\langle\omega_1,\alpha_1\rangle&\geqslant\lim\inf[\int_{S^2}(g_{+\infty}^k)^*\omega_1]+\lim\inf[\int_{S^2}(g_{-\infty}^k)^*\omega_1]\\
&\geqslant2\langle\omega_1,\alpha_1\rangle.
\end{align*}
Since $\langle\omega_1,\alpha_1\rangle>0$, we have a contradiction. So we must have
\[
\lambda_k\to\lambda\in(0,\lambda_\infty-\varepsilon']\subset(0,\lambda_\infty)
\]
and $(\ref{f2})$ still holds.

In the following, we will show that $x$ is nonconstant. Arguing indirectly let us assume $x\equiv const\in P_1\times P_2$. Denote by $v^1_k$ the $P_1-$component of $v_k$. If
\[
\int_{-\infty}^{s_k}\int_0^1(v_k^1)^*\omega_1\to0,
\]
we have $x=m_0\in\Sigma_0$ and $v_k^1\to m_0^1\in\Sigma_0^1$ uniformly. Since $h|_{\ul(\Sigma_0)}=0$, this contradicts the definition of $s_k^0$. Similarly,
\[
\int_{s_k}^{+\infty}\int_0^1(v_k^1)^*\omega_1\to0
\]
is also impossible. Therefore, we have for some $\tau>0$
\begin{equation}\label{f3}
\begin{aligned}
\int_{s_k}^{+\infty}\int_0^1(v_k^1)^*\omega_1&\geqslant\tau,\\
\int_{-\infty}^{s_k}\int_0^1(v_k^1)^*\omega_1&\geqslant\tau,\\
\int_{-\infty}^{\infty}\int_0^1(v_k^1)^*\omega_1=&\langle\omega_1,\alpha_1\rangle.
\end{aligned}
\end{equation}
Of course, since $v_k^1(\{s_k\}\times S^1)$ converges to a constant $x^1$,
by our assumption the first two integrals in $(\ref{f3})$ must be bounded below by $<\omega_1,\alpha_1>$ contradicting the equation
\[
\int_{-\infty}^{\infty}\int_0^1(v_k^1)^*\omega_1=\langle\omega_1,\alpha_1\rangle.
\]
This shows that $x$ has to be nonconstant. Eventually we have $H(x(t))\in(h_0,h_\infty)$.
This proves the theorem.
\end{proof}

\section{Applications}

We will give some applications of Theorem $\ref{p4dsgb}$ in this section.
Note that given the standard complex structure $i$ on $\cb\pb^2$ any two different points determine up to
M\"{o}bius transformation a unique holomorphic sphere $u$.
There is an embedding $u:S^2=\cb\cup\{\infty\}\hookrightarrow\cb\pb^2$ which is holomorphic.
Let $\Sigma_0^1=\{x\}$ and $\Sigma_\infty^1=\{y\}$, where $x$, $y$ are different points in $u(S^2)$.
Then with $\alpha_1=[u]$, where $u(S^2)$ is the holomorphic curve running through $x$ and $y$ we have
\[
d(\alpha_1,i,\Sigma_0^1,\Sigma_\infty^1)=[S^1]\neq[\emptyset].
\]
We note here that $i$ is a regular complex structure.
Now let $P_1=\cb\pb^2$,
$P_2$ be a SGB symplectic manifold with $[\omega_2]|_{\pi_2(P_2)}=0$,
$\Sigma_0=\{p_0\}$, $p_0\in\{x\}\times P_2$, $\Sigma_\infty=\{y\}\times P_2$.
As an application of Theorem $\ref{p4dsgb}$, we get the following corollary:

\begin{cor}\label{cp2sgb}
Let $\Sigma_0$, $\Sigma_\infty$, $P_2$ be as above, then any stable compact smooth hypersurface $\sll$ in $\cb\pb^2\times P_2$ separating $\Sigma_0$ from $\Sigma_\infty$ possesses at least one periodic Hamiltonian trajectory.
\end{cor}

It is well known that the standard cotangent bundles $(T^*N,\omega)$ over closed manifolds $N$ is SGB with $[\omega]|_{\pi_2(T^*N)}=0$
(\text{cf.} \cite{cgk2004}, \cite{lu1998}).

Liouville manifold $(\hat M, \hat\lambda)$ is a SGB symplectic manifold with $[d\hat\lambda]|_{\pi_2(\hat M)}=0$. Let us recall the definition of Liouville manifold now.
A $1$-form $\alpha$ on a manifold $\Sigma$ is called a contact form for $\xi:=\text{ker}\alpha$, if $d\alpha$ is nondegenerate on $\xi$.
In this case $\xi$ is called a contact structure.
A compact exact symplectic manifold with boundary $(M,\lambda)$ is called a {\bf Liouville domain},
if $(\Sigma:=\partial M, \alpha:=\lambda_{|\partial M})$ is a contact submanifold.
We know every Liouville domain carries a Liouville vector field $X$ defined by $\iota_X\omega=\lambda$,
and the contact condition implies that $X$ points outward at the boundary.
We can paste the positive end of a symplectization $(\Sigma\times[0,\infty),d(e^t\alpha))$ along the boundary $\Sigma$.
Then we obtain a complete {\bf Liouville manifold}, which is denoted by $(\hat M, \hat\lambda)$.

As in\cite{ag1990},\cite{wang2004}, we introduce the following notation.
\begin{defi}\label{defftt}
A noncompact manifold $M$ is said to be of finite topological type,
if there is a compact domain $\Omega\subset M$ such that $M\setminus\mathring{\Omega}$
is diffeomorphic to $\partial\Omega\times[1,\infty)$.
\end{defi}

Actually, if $M$ is a subset of a closed manifold or if $M$ is of finite topological type
the cotangent bundles $(T^*M,\omega)$ with standard symplectic structure are geometrically bounded.
This is first pointed out by Audin, Lalonde and Polterovich \cite{alp1994} P.$286$.
Lu \cite{lu1998} also claimed the cotangent bundle of a finite topological type manifold with twisted symplectic structure is SGB and omit the proof.
In the following, we will give a proof of this for the completeness of our results.
Our proof uses the idea of Proposition $2.2$ in \cite{cgk2004}.

\begin{prop}\label{psgbcotangent}
Let $M$ be a manifold of finite topological type,
then the cotangent bundle $(T^*M,\omega)$ with standard symplectic structure is SGB.
\end{prop}

\begin{proof}
Since $M$ is of finite topological type,
we may assume there is a compact domain $\Omega\subset M$ such that $M\setminus\mathring{\Omega}$
is diffeomorphic to $\partial\Omega\times[1,\infty)$. Assume the diffeomorphism is $h:\partial\Omega\times[1,\infty)\to M\setminus\mathring{\Omega}$.
Denote $\Lambda=M\setminus\mathring{\Omega}$, $\Lambda_s=h(\partial\Omega\times[s,\infty))$, $s\geqslant1$,
$\partial\Lambda_s=h(\partial\Omega\times\{s\})$.

First we will define the Riemannian metric on $T^*M$.
Let $\varphi_t$ be the flow on $T^*M$ formed by fiberwise dilations by the factor $e^t$.
Choose a fiberwise convex hypersurface $\Sigma\subset T^*M|_{\Omega}$, enclosing the compact domain $\Omega$.
Note that $\Sigma$ has contact type for $\omega$.
Let $U$ be the closure of the unbounded part of the complement to $\Sigma$ in $T^*M|_{\Omega}$.
Then $U=\cup_{t\geqslant0}\varphi_t(\Sigma).$ On the closure of the bounded part of the complement to $\Sigma$ in $T^*M|_{\Omega}$,
we can choose a compatible almost complex structure $J$.
Let $g$ be the Riemannian metric determined by $\omega$ and $J$, i.e. $g(\cdot, \cdot)=\omega( \cdot ,J\cdot)$.
(We also require that the radical vector is $g$-orthogonal to $\Sigma$.)
Now we can extend these structures to $U$ so that
\begin{equation}\label{eqgfiber}
\varphi_t^*g=e^tg~~\text{for}~t\geqslant0,
\end{equation}
i.e. g, just as $\omega$, is homogeneous of degree one with respect to the dilations, and
\[
J\circ(\varphi_t)_*=(\varphi_t)_*\circ J.
\]
Then the metric $g$, the almost complex structure $J$ and the standard symplecture $\omega$ are compatible on $U$.
Hence are compatible on $T^*M|_{\Omega}$. To define the Riemannian metric on $T^*M|_\Lambda$,
let $\bar\psi_s$ be the flow of the vector field $\partial_s$ on $\partial\Omega\times[1,\infty)$, i.e.
\begin{align*}
\bar\psi_s:~\partial\Omega\times[1,\infty)&\to\partial\Omega\times[1+s,\infty)\\
(x,s_0)&\mapsto(x,s_0+s).
\end{align*}
Let $\psi_s=h\circ\bar\psi_s\circ h^{-1}:\Lambda\to\Lambda_{1+s}$.
Then there is a natural symplectomorphism  which lifts $\psi_s$ (see \cite{s2008} Chapter 2)
\begin{align*}
\psi_{s\sharp}:T^*M|_\Lambda&\to T^*M|_{\Lambda_{1+s}}\\
(x,\xi)\mapsto&(\psi_s(x),(\psi_s^{-1})^*\xi).
\end{align*}
It is easy to see
\begin{equation}\label{eqcomm}
\psi_{s\sharp}\circ\varphi_t=\varphi_t\circ\psi_{s\sharp}.
\end{equation}
Now extend those structures to $T^*M|_\Lambda$ so that
\begin{equation}\label{eqgbase}
(\psi_{s\sharp})^*g=g~s\geqslant 1
\end{equation}
and
\[
J\circ(\psi_{s\sharp})_*=(\psi_{s\sharp})_*\circ J.
\]
We know the standard symplectic structure $\omega$ also satisfies $(\psi_{s\sharp})^*\omega=\omega$.
Then $\omega$,$J$, and $g$ are compatible on $T^*M|_\Lambda$.
Thus we get a compatible triple $(\omega,J,g)$ on $T^*M$.

The metric g is obviously complete.
Indeed, define
\begin{align*}
\Sigma_e&:=\Sigma\cup(\cup_{s\geqslant1}\psi_{s\sharp}(\Sigma\cap{T^*M|_{\partial\Omega}}))\\
\Sigma_t&:=\varphi_t(\Sigma_e).
\end{align*}
Identifying $\cup_{t\geqslant0}\Sigma_t$ with $\Sigma_e\times[0,\infty)$, the metric $g$ has the form
\begin{equation}\label{gformula}
g(\cdot,\cdot)=e^{t}(g(\partial_t,\partial_t)dt^2+g|_{\Sigma_e})((\varphi_{t}^{-1})_*\cdot,(\varphi_{t}^{-1})_*\cdot)\circ\varphi_{t}^{-1}.
\end{equation}
It is clear the integral curves $\varphi_t(x)$, for $t>0$ and $x\in\Sigma_e$, are minimizing geodesics of g.
The distance from $x$ to $\varphi_t(x)$, $L_x(t)=\int_0^t(e^tg(\partial_t,\partial_t))^{\frac{1}{2}}dt$, goes to $\infty$ as $t\to\infty$.
Let $|s_1-s_2|$ be positive and small. Assume $x\in T^*M|_{\partial\Lambda_{s_1}}\cap\Sigma_{t_1}$ and $y\in T^*M|_{\partial\Lambda_{s_2}}\cap\Sigma_{t_2}$.
If $|t_2-t_1|\to\infty$, dist$(x,y)\to\infty$. Let $\gamma(s)$ be a curve with $\gamma(0)=x$, $\gamma(1)=y$.
From $(\ref{gformula})$, we have the length $L(\gamma(s))$ of $\gamma(s)$ equals $e^{\frac{t_0}{2}}L(\varphi_{t_0}^{-1}(\gamma(s)))$,
where $L(\varphi_{t_0}^{-1}(\gamma(s)))$ is the length of $\varphi_{t_0}^{-1}(\gamma(s))$.
Thus we can get $\text{dist}(T^*M|_{\partial\Lambda_{s_1}},T^*M|_{\partial\Lambda_{s_2}})$
is determined by the compact parts of $T^*M|_{\partial\Lambda_{s_1}}$ and $T^*M|_{\partial\Lambda_{s_2}}$.
Thus 
$$
\text{dist}(T^*M|_{\partial\Lambda_{s_1}},T^*M|_{\partial\Lambda_{s_2}})>0.
$$
From equation $(\ref{eqgbase})$, we know a curve $\gamma(t)$ from $T^*M|_{\partial\Lambda_{s_1}}$ to $T^*M|_{\partial\Lambda_{s_2}}$
has the same length with the curve $\psi_{s\sharp}(\gamma(t))$ from
$T^*M|_{\partial\Lambda_{s+s_1}}$ to $T^*M|_{\partial\Lambda_{s+s_2}}$. $\psi_{s\sharp}$ is a symplectomorphism.
Thus we have
\[
\text{dist}(T^*M|_{\partial\Lambda_{s+s_1}},T^*M|_{\partial\Lambda_{s+s_2}})=
\text{dist}(T^*M|_{\partial\Lambda_{s_1}},T^*M|_{\partial\Lambda_{s_2}}).
\]
Therefore, every bounded subset of $T^*M$ is contained in a compact subset and is relatively compact.
By Hopf-Rinow Theorem, this is equivalent to completeness.

From the Lemma $1$ in \cite{g1978} and the definition of the metric $(\ref{eqgfiber})$,
it follows that the sectional curvature of $g$ goes to zero as $x\to\infty$ in $U$.
Thus the sectional curvature of $g$ is bounded from above on $T^*M|_{\Omega}$.
From $(\ref{eqgbase})$ we know the sectional curvature of $g$ on $T^*M|_\Lambda$ is determined by the sectional curvature of $g$ on $T^*M|_\Omega$
which is bounded from above.
We get that the sectional curvature of $g$ is bounded from above on $T^*M$.

Define $\Sigma':=\varphi_1(\Sigma\cup(\cup_{s\leqslant2}\psi_{s\sharp}(\Sigma\cap{T^*M|_{\partial\Omega}})))$.
Let $W,U'$ denote the closure of the bounded and unbounded part of the complement to
$\Sigma'$ in $T^*M|_{\Omega\cup h(\partial\Omega\times[1,2])}$ respectively.
Since $W$ is compact, we know the injectivity radius of it is bounded away from zero.
By Lemma $\ref{lemgeodesics}$, we know a curve $\gamma_s(x)$ through $x\in U_{\frac{1}{2}}$,$U_\frac{1}{2}=\cup_{t\geqslant\frac{1}{2}}\Sigma_t$, is a geodesic
if and only if $\varphi_{t}^{-1}(\gamma_s(x))$ is a geodesic through $\varphi_{t}^{-1}(x)$, for any $0<t\leqslant t_0(x)$,
where $t_0(x)$ is the real number such that $\varphi_{t_0}^{-1}(x)\in\Sigma_{\frac{1}{2}}$.
A curve $\gamma_t(x)$ is a geodesic through $x\in T^*M|_{\Lambda_2}$
if and only if $\psi_{s\sharp}^{-1}(\gamma_t(x))$ is a geodesic through $\psi_{s\sharp}^{-1}(x)$, for any $0<s\leqslant s_0(x)$,
$s_0(x)$ is the real number such that $\psi_{s_0\sharp}^{-1}(x)\in T^*M|_{\partial\Lambda_{\frac{3}{2}}}$.
Thus if $\gamma(x)$ is a geodesic loop with small length $L(\gamma)$ through $x\in T^*M\setminus W$,
there is a geodesic loop $\gamma'$ in $W$ with length $L(\gamma')\leqslant L(\gamma)$.
Combined with the completeness and the upper bound of sectional curvature,
we know if the injectivity radius of metric $g$ on $T^*M$ is not positive there is a geodesic loop $\gamma'$ in $W$
with length as small as we want (Lemma $16$ in \cite{p2006} P.$142$). This contradicts to the fact that the injectivity radius on $W$ is positive.
Thus we have the injectivity radius of $g$ is bounded away from zero on $T^*M$.
\end{proof}

The calculus of the geodesics is given in the following lemma.
\begin{lem}\label{lemgeodesics}
With the metric $g$ and notations defined in Proposition $\ref{psgbcotangent}$,
a curve $\gamma_s(x)$ through $x\in U_\frac{1}{2}$,$U_\frac{1}{2}=\cup_{t\geqslant\frac{1}{2}}\Sigma_t$, is a geodesic
if and only if $\varphi_{t}^{-1}(\gamma_s(x))$ is a geodesic through $\varphi_{t}^{-1}(x)$, for any $0<t\leqslant t_0(x)$,
where $t_0(x)$ is the real number such that $\varphi_{t_0}^{-1}(x)\in\Sigma_\frac{1}{2}$.
A curve $\gamma_t(x)$ is a geodesic through $x\in T^*M|_{\Lambda_\frac{3}{2}}$
if and only if $\psi_{s\sharp}^{-1}(\gamma_t(x))$ is a geodesic through $\psi_{s\sharp}^{-1}(x)$, for any $0<s\leqslant s_0(x)$,
where $s_0(x)$ is the real number such that $\psi_{s_0\sharp}^{-1}(x)\in T^*M|_{\partial\Lambda_{\frac{3}{2}}}$.
\end{lem}

\begin{proof}
We only give the proof of the first assertion here, since the second can be proved similarly.
Let $\pi:T^*M\to M$ be the projection of the cotangent bundle.
Assume $x\in U_\frac{1}{2}$ such that $\pi(x)\in\Omega$. The metric is defined by $(\varphi_t)^*g=e^tg$. Thus we have
\begin{equation*}
\begin{aligned}
g&=(\varphi_{t}^{-1})^*(e^tg)\\
g(\cdot,\cdot)&=e^tg((\varphi_{t}^{-1})_*\cdot,(\varphi_{t}^{-1})_*\cdot)\circ\varphi_{t}^{-1}.
\end{aligned}
\end{equation*}
Choose a local coordinate chart $(V,\varphi)$ of $M$ such that $\pi(x)\in V$ and $(\pi^{-1}(V),h')$ is a local trivialization of $T^*M$,
i.e.
\[
\pi^{-1}(V)\to V\times\rb^n\to\varphi(V)\times\rb^n
\]
is a local coordinate chart of $x$ in $T^*M$. Let $(x^1,...,x^n,y^1,...y^n)$ be the local coordinates
and denote
\begin{equation*}
\begin{aligned}
\partial_i=\frac{\partial}{\partial x^i}~i=1,2,...,n,~
&\partial_{\bar i}=\frac{\partial}{\partial y^i}~i=1,2,...,n.
\end{aligned}
\end{equation*}
Let
\begin{align*}
g_{ij}=g(\partial_i,\partial_j),~ &g_{\bar i,\bar j}=g(\partial_{\bar i},\partial_{\bar j}),\\
g_{i,\bar j}=g(\partial_i,\partial_{\bar j}),~&g_{\bar i,j}=g(\partial_{\bar i},\partial_j).
\end{align*}
Then
\begin{align*}
\begin{pmatrix}
(g^{ij})~&(g^{i\bar j})\\
(g^{\bar ij})~&(g^{\bar i\bar j})
\end{pmatrix}
=&
\begin{pmatrix}
(g_{ij})~&(g_{i\bar j})\\
(g_{\bar ij})~&(g_{\bar i\bar j})
\end{pmatrix}
^{-1}.
\end{align*}
The Christoffel symbols corresponding to the Riemannian metric $g$ is given by
\[
\Gamma_{ij}^k=\frac{1}{2}g^{k\xi}(\partial_j g_{i\xi}+\partial_i g_{j\xi}-\partial_\xi g_{ij})
+\frac{1}{2}g^{k\bar\xi}(\partial_j g_{i\bar\xi}+\partial_i g_{j\bar\xi}-\partial_{\bar\xi} g_{ij}).
\]
The push forward of the vector fields can be given by
\[
(\varphi_{t}^{-1})_*
\begin{pmatrix}
\partial_1 \\
\vdots\\
\partial_n\\
\partial_{\bar1}\\
\vdots\\
\partial_{\bar n}
\end{pmatrix}
\Bigg|_{x}
=
\begin{pmatrix}
Id~&0\\
0~&e^{-t}Id
\end{pmatrix}
\begin{pmatrix}
\partial_1 \\
\vdots\\
\partial_n\\
\partial_{\bar1}\\
\vdots\\
\partial_{\bar n}
\end{pmatrix}
\Bigg|_{{\varphi_{t}^{-1}(x)}}.
\]
Then we have
\[
\begin{pmatrix}
(g_{ij})~&(g_{i\bar j})\\
(g_{\bar ij})~&(g_{\bar i\bar j})
\end{pmatrix}
=
\begin{pmatrix}
(e^tg_{ij})~&(g_{i\bar j})\\
(g_{\bar ij})~&(e^{-t}g_{\bar i\bar j})
\end{pmatrix}
\circ\varphi_{t}^{-1}.
\]
Thus
\[
\begin{pmatrix}
(g^{ij})~&(g^{i\bar j})\\
(g^{\bar ij})~&(g^{\bar i\bar j})
\end{pmatrix}
=
\begin{pmatrix}
(e^{-t}g^{ij})~&(g^{i\bar j})\\
(g^{\bar ij})~&(e^tg^{\bar i\bar j})
\end{pmatrix}
\circ\varphi_{t}^{-1}.
\]
To get the relation of $\Gamma_{ij}^k$ with $\Gamma_{ij}^k\circ\varphi_{t}^{-1}$, we need the following relation
\begin{equation*}
\begin{aligned}
\partial_ig_{jk}&=\partial_ig(\partial_j,\partial_k)\\
&=\partial_i[e^tg((\varphi_{t}^{-1})_*\partial_j,(\varphi_{t}^{-1})_*\partial_k)\circ\varphi_{t}^{-1}]\\
&=e^t[(\partial_lg_{jk})\circ\varphi_{t}^{-1}\partial_i(\varphi_{t}^{-1})^l+
(\partial_{\bar l}g_{jk})\circ\varphi_{t}^{-1}\partial_i(\varphi_{t}^{-1})^{\bar l}]\\
&=e^t(\partial_ig_{jk})\circ\varphi_{t}^{-1}.
\end{aligned}
\end{equation*}
Similarly, we have
\begin{equation*}
\begin{aligned}
&\partial_ig_{\bar jk}=(\partial_ig_{\bar jk})\circ\varphi_{t}^{-1},~
\partial_ig_{j\bar k}=(\partial_ig_{j\bar k})\circ\varphi_{t}^{-1},\\
&\partial_{\bar i}g_{jk}=(\partial_{\bar i}g_{jk})\circ\varphi_{t}^{-1},~
\partial_ig_{\bar j\bar k}=e^{-t}(\partial_ig_{\bar j\bar k})\circ\varphi_{t}^{-1},\\
&\partial_{\bar i}g_{\bar jk}=e^{-t}(\partial_{\bar i}g_{\bar jk})\circ\varphi_{t}^{-1},~
\partial_{\bar i}g_{j\bar k}=e^{-t}(\partial_{\bar i}g_{j\bar k})\circ\varphi_{t}^{-1},\\
&\partial_{\bar i}g_{\bar j\bar k}=e^{-2t}(\partial_{\bar i}g_{\bar j\bar k})\circ\varphi_{t}^{-1}.
\end{aligned}
\end{equation*}
The Christoffel symbols $\Gamma_{ij}^k$ and $\Gamma_{ij}^k\circ\varphi_{t}^{-1}$ have the following relation
\begin{equation*}
\begin{aligned}
\Gamma_{ij}^k=&\frac{1}{2}g^{k\xi}(\partial_j g_{i\xi}+\partial_i g_{j\xi}-\partial_\xi g_{ij})
+\frac{1}{2}g^{k\bar\xi}(\partial_j g_{i\bar\xi}+\partial_i g_{j\bar\xi}-\partial_{\bar\xi} g_{ij})\\
=&\frac{1}{2}e^{-t}g^{k\xi}\circ\varphi_{t}^{-1}
(e^t(\partial_jg_{i\xi})\circ\varphi_{t}^{-1}+e^t(\partial_ig_{j\xi})\circ\varphi_{t}^{-1}-e^t(\partial_\xi g_{ij})\circ\varphi_{t}^{-1})\\
&+\frac{1}{2}g^{k,\bar\xi}\circ\varphi_{t}^{-1}
((\partial_jg_{i\bar\xi})\circ\varphi_t^{-1}+(\partial_ig_{j\bar\xi})\circ\varphi_t^{-1}-(\partial_{\bar\xi} g_{ij})\circ\varphi_t^{-1})\\
=&\frac{1}{2}g^{k\xi}\circ\varphi_t^{-1}
((\partial_jg_{i\xi})\circ\varphi_t^{-1}+(\partial_ig_{j\xi})\circ\varphi_t^{-1}-(\partial_\xi g_{ij})\circ\varphi_t^{-1})\\
&+\frac{1}{2}g^{k,\bar\xi}\circ\varphi_t^{-1}
((\partial_jg_{i\bar\xi})\circ\varphi_t^{-1}+(\partial_ig_{j\bar\xi})\circ\varphi_t^{-1}-(\partial_{\bar\xi}g_{ij})\circ\varphi_t^{-1})\\
=&\Gamma_{ij}^k\circ\varphi_t^{-1}.
\end{aligned}
\end{equation*}
Similarly, we have
\begin{equation*}
\begin{aligned}
&\Gamma_{ij}^{\bar k}=e^t\Gamma_{ij}^{\bar k}\circ\varphi_t^{-1},~\Gamma_{\bar ij}^{k}=e^{-t}\Gamma_{\bar ij}^{k}\circ\varphi_t^{-1}\\
&\Gamma_{i\bar j}^{k}=e^{-t}\Gamma_{i\bar j}^{k}\circ\varphi_t^{-1},~\Gamma_{\bar ij}^{\bar k}=\Gamma_{\bar ij}^{\bar k}\circ\varphi_t^{-1}\\
&\Gamma_{i\bar j}^{\bar k}=\Gamma_{i\bar j}^{\bar k}\circ\varphi_t^{-1},~\Gamma_{\bar i\bar j}^{k}=e^{-2t}\Gamma_{\bar i\bar j}^{k}\circ\varphi_t^{-1}\\
&\Gamma_{\bar i\bar j}^{\bar k}=e^{-t}\Gamma_{\bar i\bar j}^{\bar k}\circ\varphi_t^{-1}.
\end{aligned}
\end{equation*}
Now suppose $\gamma_s(x)$ is a curve through $x$. In local coordinates $\gamma_s(x)$ is given by
\[
\gamma_s(x)=(\gamma^1(s),\ldots,\gamma^{n}(s),\gamma^{\bar 1}(s),\ldots,\gamma^{\bar n}(s)).
\]
Then $\varphi_t^{-1}(\gamma_s(x))$ is given by
\[
\varphi_t^{-1}(\gamma_s(x))=(\gamma^1(s),\ldots,\gamma^{n}(s),e^{-t}\gamma^{\bar 1}(s),\ldots,e^{-t}\gamma^{\bar n}(s)).
\]
Equation of geodesics in the local coordinates
\begin{equation*}
\begin{aligned}
&\frac{d^2\gamma^k}{ds}+\Gamma^k_{ij}\circ\varphi_t^{-1}\frac{d\gamma^i}{ds}\frac{d\gamma^j}{ds}
+\Gamma^k_{\bar ij}\circ\varphi_t^{-1}e^{-t}\frac{d\gamma^{\bar i}}{ds}\frac{d\gamma^j}{ds}\\
&~~+\Gamma^k_{i\bar j}\circ\varphi_t^{-1}\frac{d\gamma^{i}}{ds}e^{-t}\frac{d\gamma^{\bar j}}{ds}
+\Gamma^k_{\bar i\bar j}\circ\varphi_t^{-1}e^{-t}\frac{d\gamma^{\bar i}}{ds}e^{-t}\frac{d\gamma^{\bar j}}{ds}\\
=&\frac{d^2\gamma^k}{ds}+\Gamma^k_{ij}\frac{d\gamma^i}{ds}\frac{d\gamma^j}{ds}
+\Gamma^k_{\bar ij}\frac{d\gamma^{\bar i}}{ds}\frac{d\gamma^j}{ds}
+\Gamma^k_{i\bar j}\frac{d\gamma^{i}}{ds}\frac{d\gamma^{\bar j}}{ds}
+\Gamma^k_{\bar i\bar j}\frac{d\gamma^{\bar i}}{ds}\frac{d\gamma^{\bar j}}{ds}\\
=&0.~
\end{aligned}
\end{equation*}
We know $\gamma_s(x)$ is a geodesic if and only if $\varphi_t^{-1}(\gamma_s(x))$ is a geodesic.

Now if $x\in U_\frac{1}{2}$ such that $\pi(x)\in \Lambda$, we can prove the first assertion in a similar way.
Indeed, from equation $(\ref{eqcomm})$ we know
\begin{align*}
(\psi_{s\sharp})^*\varphi_t^*g&=\varphi_t^*(\psi_{s\sharp})^*g,\\
(\psi_{s\sharp})^*\varphi_t^*g&=e^t(\psi_{s\sharp})^*g,\\
\varphi_t^*g&=e^tg,\\
g&=e^t(\varphi_t^{-1})^*g.
\end{align*}
\end{proof}
From Corollary $\ref{cp2sgb}$ and Proposition $\ref{psgbcotangent}$, it is easy to get Corollary $\ref{cormain}$

\label{lastpage}
\end{document}